\documentclass[11pt]{amsart}

\usepackage{amssymb}
\usepackage{amsmath}
\usepackage{amsthm}

\theoremstyle{plain}

\newtheorem{theorem}{Theorem}[section]

\newtheorem{lemma}[theorem]{Lemma}

\newtheorem{corollary}[theorem]{Corollary}

\newtheorem{fact}[theorem]{Fact}

\theoremstyle{definition}

\newtheorem{definition}[theorem]{Definition}
\newtheorem{remark}[theorem]{Remark}
\newtheorem{example}[theorem]{Example}

\newtheorem*{problem*}{Problem}
\newtheorem*{claim}{Claim}

\newcommand{\dist}{\mathrm{dist}}
\newcommand{\C}{\mathrm{c}}

\newcommand{\SX}{S_{X^*}}

\newcommand{\R}{\mathbb{R}}
\newcommand{\N}{\mathbb{N}}

\newcommand{\inte}{\mathrm{int}\,}
\newcommand{\cconv}{\overline{\mathrm{conv}}\,}

\newcommand{\conv}{{\mathrm{conv}}\,}

\renewcommand{\epsilon}{\varepsilon}
\renewcommand{\phi}{\varphi}

\begin{document}

\title{A variational approach to the alternating projections method}

\author{Carlo Alberto De Bernardi
}
\address{Dipartimento di Matematica per le Scienze economiche, finanziarie ed attuariali, Universit\`{a} Cattolica del Sacro Cuore, Via Necchi 9, 20123 Milano, Italy}

\email{carloalberto.debernardi@unicatt.it, carloalberto.debernardi@gmail.com}

\author{Enrico Miglierina
}
\address{Dipartimento di Matematica per le Scienze economiche, finanziarie ed attuariali, Universit\`{a} Cattolica del Sacro Cuore, Via Necchi 9, 20123 Milano, Italy}

\email{enrico.miglierina@unicatt.it}



 \thanks{
}



\begin{abstract}
	The 2-sets convex feasibility problem aims at finding a point in
	the nonempty intersection of two closed convex sets $A$ and $B$ in a Hilbert space $X$. The method of alternating projections is the simplest iterative procedure for finding a solution and it goes back to von Neumann. 	
	In the present paper, we study some stability properties for this method in the following sense: we consider two sequences of sets,
	each of them converging, with respect to the  Attouch-Wets variational convergence, respectively, to $A$ and $B$. Given a starting point $a_0$, we consider the sequences of points obtained by projecting on the ``perturbed'' sets, i.e., the sequences $\{a_n\}$ and $\{b_n\}$ given by $b_n=P_{B_n}(a_{n-1})$ and $a_n=P_{A_n}(b_n)$.  	
	Under appropriate geometrical and topological
	assumptions on the intersection of the limit sets, we ensure that the sequences $\{a_n\}$ and $\{b_n\}$ 
 converge in norm to a point in the intersection of $A$ and $B$. In particular, we consider both when the intersection $A\cap B$ reduces to a singleton and when the interior of $A \cap B$ is nonempty. Finally we consider the case in which the limit sets $A$ and $B$ are subspaces.
\end{abstract}

\maketitle

\section{Introduction}
The 2-sets convex  feasibility problem is the classical problem of finding
a point in the nonempty intersection of two closed and
convex sets $A$ and $B$ in a Hilbert space $X$ (see \cite[Section~4.5]{BorweinZhu} for some basic
results on this subject). Many efforts have been devoted to the study of algorithmic
procedures to solve convex feasibility problems, both from a
theoretical and from a computational point of view  (see, e.g.,
\cite{BauschkeBorwein,BCOMB,BorweinSimsTam,Censor,Hundal} and the 
references therein).
The method of alternating projections is the simplest iterative procedure for finding a solution and it goes back to von Neumann \cite{vonNeumann}: let us denote by $P_A$ and $P_B$ the projections on the sets $A$ and $B$, respectively, and, given a starting point $c_0\in X$, consider the {\em alternating projections sequences} $\{c_n\}$ and $\{d_n\}$ given  by $$d_n=P_{B}(c_{n-1})\ \ \text{and}\ \ c_n=P_{A}(d_n)\ \ \ \ \ (n\in\N).$$
In the case the sequences $\{c_n\}$ and $\{d_n\}$ converge in norm to a point in the intersection of $A$ and $B$, we say that the method of alternating projections converges.

Many concrete problems in applications
can be formulated as a convex feasibility problem. As typical
examples, we mention solution of convex inequalities, partial
differential equations, minimization of convex nonsmooth
functions, medical imaging, computerized tomography and image
reconstruction. For some details and other applications see, e.g.,
\cite{BauschkeBorwein} and the references therein.

Often in concrete applications data are affected by some
uncertainties. Hence stability  of  solutions of a convex feasibility problem with respect to data
perturbations is a desirable property, both from theoretical and computational point of view. In the present paper we investigate some
``stability'' properties of the alternating projections method in the following sense. Let us suppose that 
$\{A_n\}$ and $\{B_n\}$
are two sequences of closed convex sets such that $A_n\rightarrow
A$ and $B_n\rightarrow B$ for the Attouch-Wets {variational} convergence (see Definition~\ref{def:AW}) and let us introduce the definition of {\em perturbed alternating projections sequences}.

\begin{definition}\label{def:perturbedseq} Given $a_0\in X$, the {\em perturbed alternating projections sequences}  $\{a_n\}$ and $\{b_n\}$, w.r.t. $\{A_n\}$ and $\{B_n\}$ and with starting point $a_0$, are defined inductively by
	$$b_n=P_{B_n}(a_{n-1})\ \ \ \text{and}\ \ \  a_n=P_{A_n}(b_n) \ \ \ \ \ \ \ \ \ (n\in\N)$$ 
\end{definition}

\noindent Our aim is to find some conditions on the limit sets $A$ and $B$ that guarantee, for each choice of the sequences $\{A_n\}$ and $\{B_n\}$ and for each choice of  the starting point $a_0$, the convergence in norm
of the corresponding perturbed alternating projections sequences  $\{a_n\}$ and $\{b_n\}$. If this is the case, we say that the couple $(A,B)$ is {\em stable}.

The results reported in this paper can be seen as a continuation of  the research considered in \cite{DebeMiglMol}. However, compared with the notion of stability  studied in that paper,  the approach developed here seems to be more interesting also from a computational point of view since it does not require to find an exact solution of the ``perturbed problems'' (i.e. the problems given by the sets $A_n$ and $B_n$) but only to consider projections on the ``perturbed'' sets  $A_n$ and $B_n$. Moreover, the techniques used in the proofs are completely different from those of \cite{DebeMiglMol}. 

\smallskip

Clearly, in order  that the couple $(A,B)$ is stable, it is necessary that the alternating projections sequences $\{c_n\}$ and $\{d_n\}$ converge in norm (indeed, we can consider the particular case in which the sequences of sets $\{A_n\}$ and $\{B_n\}$ are given by $A_n=A$ and $B_n=B$, whenever $n\in\N$). Since, in general, this is not the case (see \cite{Hundal,MatouReich}), we shall restrict our attention to those situations in which the method of alternating projections converges. After some preliminaries, contained in Section~\ref{Notations and preliminaries}, we consider, in Sections~\ref{section:Infinite-dimensional}, \ref{Section:nonempty interior} and \ref{section:subspaces}, respectively, the following three cases:
\begin{enumerate}
	\item $A$ and $B$ are separated by a strongly exposing functional $f$ for the set $A$, i.e.,  there exist $x_0\in A\cap B$ and a linear continuous functional $f$ such that $\inf f(B)=f(x_0)=\sup f(A)$  and such that $f$ strongly exposes $A$ at $x_0$ (see Definition~\ref{def:strexp});
\item the intersection between $A$ and $B$ has nonempty interior;  
\item $A$ and $B$ are closed subspaces.
\end{enumerate}

Observe that if (i) is satisfied then the method of alternating projections converges. Indeed, by \cite[Lemma~4.5.11]{BorweinZhu} or by \cite[Theorem~1.4]{KopReichHilbert},  the alternating projections sequences $\{c_n\}$ and $\{d_n\}$ satisfy $\|c_n-d_n\|\to0$. Then it is easy to verify that $f(c_n),f(d_n)\to f(x_0)$ and hence, since $f$ strongly exposes $A$ at $x_0$, we have that $c_n,d_n\to x_0$ in norm. 

Similar assumption on the limit sets has been considered by the authors and E.~Molho in the recent paper \cite{DebeMiglMol}, in which they proved, among other things, that if (i) is satisfied and if $x_n\in A_n,\,y_n\in B_n$ are such that
$\|x_n-y_n\|$ coincides with the distance between $A_n$ and $B_n$ then   $x_n,y_n\to x_0$ in norm (see the proof of \cite[Theorem~4.5]{DebeMiglMol}).
In Section~\ref{section:Infinite-dimensional} of the present paper, we prove that if $A$ and $B$ are separated by a strongly exposing functional $f$ for the set $A$ then, for each choice of sequences $\{A_n\},\,\{B_n\}$ and starting point $a_0$, the corresponding perturbed alternating projections sequences  $\{a_n\}$ and $\{b_n\}$ converge in norm to $x_0$ (cf. Theorem~\ref{puntolur} below). In this case, our approach  is essentially based on suitable approximations of the sets $A_n$ and $B_n$ by convex and non-convex cones, respectively.
This result shed a new light also on the celebrated example of Hundal (see \cite{Hundal}) of a convex feasibility problem in a Hilbert space whose corresponding alternating projections sequences do not norm converge. There, $A$ is a convex cone and $B$ is a hyperplane touching the vertex of the  cone $A$; this hyperplane is defined by a functional that does not strongly expose the vertex of the cone. Our result prove that, if we consider a hyperplane defined by a functional strongly exposing the vertex of the cone, we obtain not only the norm convergence of the alternating projections, but also the convergence of the perturbed alternating projections, i.e., the couple $(A,B)$ is stable.
  

{In
	Section~\ref{Section:nonempty interior}, we investigate to what extent it is possible to guarantee convergence of the perturbed alternating projections in the case $A\cap B$ is nonempty but does not reduce to a singleton. 
	Example~\ref{ex: notconverge} show that, in general, even in the finite-dimensional setting and even if $A\cap B$ is bounded,  the couple $(A,B)$ may be not stable. On the other hand, Theorem~\ref{theorem:corpilur} ensures that the couple $(A,B)$ is stable whenever $\mathrm{int}\,(A\cap B)\neq \emptyset$. We point out that boundedness of $A \cap B$ is not required. Moreover, we apply the results of this section to investigate the convergence of perturbed alternating projections for the inequality constraints problem.
}

Finally the last section of the paper is devoted to the case (iii) where $A$ and $B$ are closed subspaces. The convex feasibility problem where $A$ and $B$ are subspaces  is the original problem studied  by von~Neumann. In his, now classical, theorem (see  \cite{vonNeumann}), he proved that the alternating projections sequences $\{c_n\}$ and $\{d_n\}$ converge in norm to $P_{A\cap B}(a_0)$. This theorem was rediscovered by several authors and many alternative proofs were provided (see, e.g., \cite{KopReich,KopReichHilbert} and the references therein). 
In Section~\ref{section:subspaces}, we study the problem of convergence of perturbed alternating projections sequences in the case in which $A$ and $B$ are subspaces.
Example~\ref{ex:duedimensionale} below shows that  even in the finite-dimensional setting it is conceivable that the perturbed projections sequences are unbounded in the case $A\cap B\neq\{0\}$. For this, in Section~\ref{section:subspaces}, we focus on the situation in which $A$ and $B$ are closed subspaces such that $A\cap B=\{0\}$. It turns out that if $A+B$ is a closed subspace then the couple $(A,B)$ is stable (Theorem~\ref{prop:sottospazisommachiusa}). On the other hand,  in Theorem~\ref{teo:sommaNONchiusa}, we provide a couple $(A,B)$ of closed subspaces such that $A\cap B=\{0\}$ and such that there exist sequences of sets $\{A_n\},\,\{B_n\}$ and starting point $a_0$ such that the corresponding perturbed projections sequences are unbounded. Our construction is based on the example, contained in \cite{FranchettiLight86}, of two subspaces of a Hilbert space with non-closed sum  such that the convergence of the corresponding alternating projection method is not geometric (for the definition of geometric convergence see \cite{FranchettiLight86}, see also  \cite{PusReichZas} for some results concerning the convergence rate of the alternating projection algorithm for the case of $n$  subspaces).

\section{Notations and preliminaries}\label{Notations and preliminaries}

Throughout all this paper, if not differently stated, $X$ denotes a real normed space with
the topological dual $X^*$. We
denote by $B_X$ and $S_X$ the closed unit ball and the unit sphere of $X$, respectively. 
For $x,y\in X$, $[x,y]$ denotes the closed segment in $X$ with
endpoints $x$ and $y$.
For a subset $K$ of $X$,
$\alpha>0$, and a functional $f\in \SX$ bounded on $K$, let
$$S(f,\alpha,K)=\{x\in
K;\, f(x)\geq\sup f(K)-\alpha\}$$ be the closed slice of $K$
given by $\alpha$ and $f$.

For $f\in\SX$ and $\alpha\in(0,1)$, we denote
$$C(f,\alpha)=\{x\in X; f(x)\geq \alpha\|x\|\},\
V(f,\alpha)=\{x\in X; f(x)\leq\alpha\|x\|\}.$$
It is easy to see that $C(f,\alpha)$ and
$V(f,\alpha)$ are nonempty closed cones and that $C(f,\alpha)$ is convex.

For a subset $A$ of $X$, we denote by $\inte(A)$, $\partial A$, $\conv(A)$ and
$\cconv(A)$ the interior, the boundary, the convex hull and the closed convex
hull of $A$, respectively. 
We
denote by $$\textstyle \mathrm{diam}(A)=\sup_{x,y\in A}\|x-y\|,$$
the (possibly infinite) diameter of $A$. For $x\in X$, let
$$\dist(x,A) =\inf_{a\in A} \|a-x\|.$$ Moreover, given $A,B$
nonempty subsets of $X$,  we denote by $\dist(A,B)$ the usual
``distance'' between $A$ and $B$, that is,
$$ \dist(A,B)=\inf_{a\in A} \dist(a,B).$$

\medskip 

Let us now introduce some definitions and basic properties concerning convergence of sets. By $\C(X)$ we denote the family of all nonempty closed subsets of
$X$. 
%
%
%
Let us introduce the (extended) Hausdorff metric $h$ on
$\C(X)$. For $A,B\in\C(X)$, we define the excess of $A$ over $B$
as

$$e(A,B) = \sup_{a\in A} \mathrm{dist}(a,B).$$

\noindent Moreover, if $A\neq\emptyset$ and $B=\emptyset$ we put
$e(A,B)=\infty$,  if $A=\emptyset$ we put $e(A,B)=0$. For $A,B\in \C(X)$, we define

$$h(A,B)=\max \bigl\{ e(A,B),e(B,A) \bigr\}.$$

\begin{definition} A sequence $\{A_j\}$ in $\C(X)$ is said to
Hausdorff converge to $A\in\C(X)$ if $$\textstyle \lim_j h(A_j,A)
= 0.$$
\end{definition}


Next we recall the definition of the so called Attouch-Wets convergence (see,
e.g., \cite[Definition~8.2.13]{LUCC}), which can be seen as a
localization of the Hausdorff convergence.  If $N\in\N$ and
$A,C\in\C(X)$, define
\begin{eqnarray*}
e_N(A,C) &=& e(A\cap N B_X, C)\in[0,\infty),\\
h_N(A,C) &=& \max\{e_N(A,C), e_N(C,A)\}.
\end{eqnarray*}

\begin{definition}\label{def:AW} A sequence $\{A_j\}$ in $\C(X)$ is said to
Attouch-Wets converge to $A\in\C(X)$ if, for each $N\in\N$,
$$\textstyle \lim_j h_N(A_j,A)= 0.$$
\end{definition}

Several times without mentioning it, we shall use
 the following two results.

\begin{theorem}[{see, e.g., \cite[Theorem~8.2.14]{LUCC}}]  The sequence of sets $\{A_n\}$  Attouch-Wets converges to $A$ if{f}
$$\textstyle \sup_{\|x\|\leq N}{|\mathrm{dist}(x,A_n)-\mathrm{dist}(x,A)|}\to 0
\ \ \ (n\to \infty),$$ 
 whenever $N\in\N$.
\end{theorem}

\begin{fact}\label{fact:AW} Let $A$ be a nonempty closed convex set in a Banach space $X$. Suppose that 
	$\{A_n\}$ is a sequence of closed convex sets such that $A_n\rightarrow
	A$ for the Attouch-Wets convergence. Then, if $\{a_n\}$ is a bounded sequence in $X$ such that $a_n\in A_n$ ($n\in\N$), we have that $\mathrm{dist}(a_n,A)\to 0$.
	
\end{fact}

\medskip

\begin{definition}[{see, e.g., \cite[Definition~7.10]{FHHMZ}}]\label{def:strexp} Let $A$ be a nonempty subset of a normed space $X$. A point $a\in A$
	is called a strongly exposed point of $A$ if there exists  a
	support functional $f\in X^*\setminus\{0\}$ for $A$ in $a$ $\bigl($i.e.,
	$f (a) = \sup f(A)$$\bigr)$, such that  $x_n\to a$ for all sequences
	$\{x_n\}$ in $A$ such that $\lim_n f(x_n) = \sup f(A)$. In this
	case, we say that $f$ strongly exposes $A$ at $a$.
\end{definition}

\noindent Let us observe that $f\in S_{X^*}$ strongly exposes $A$ at $a$
if{f} $f(a)=\sup f(A)$ and
$$\mathrm{diam}\bigl(S(f,\alpha,A)\bigr)\to0 \text{ as } \alpha\to 0^+.$$

\noindent Let us
recall that a {\em body} in $X$ is a closed convex set in $X$ with
nonempty interior.

\begin{definition}[{see, e.g., \cite[Definition~1.3]{KVZ}}]
	Let $A\subset X$ be a body. We say that $x\in\partial A$ is an
	{\em LUR (locally uniformly rotund) point} of $A$ if for each
	$\epsilon>0$ there exists $\delta>0$ such that if $y\in A$
	and $\dist(\partial A,(x+y)/2)<\delta$ then $\|x-y\|<\epsilon$. 
\end{definition}

If $A=B_X$, the previous definition coincides with the standard definition of
local uniform rotundity of the norm at $x$.
We say that $A$ is an {\em LUR body} if each point in
$\partial A$ is an LUR point of $A$.

\begin{lemma}\label{slicelimitatoselur} Let $A$ be a body in $X$
	and suppose that $a\in\partial A$ is an LUR point of $A$. Then, if
	$f\in S_{X^*}$ is a support functional for $A$ in $a$, $f$
	strongly exposes $A$ at $a$.
\end{lemma}

The lemma is well-known in the case the body is
a ball (see, e.g., \cite[Exercise~8.27]{FHHMZ}) and in the general
case the proof is similar (see, e.g., \cite[Lemma~4.3]{DebeMiglMol}).

The next lemma gives a characterization of those functionals $f$ that strongly expose a set $A$ in terms of containment of $A$ in translations of cones of the form $C(f,\alpha)$.  

\begin{lemma}\label{lemma:stronglyexposedVSopiccolo}
	Let $A$ be a convex set in $X$ such that $0\in A$. Let $f\in S_{X^*}$ be such that 
	$f (0) = \inf f(A)$ and let $x_0\in S_X$ be such that $f(x_0)=1$ . Let us consider $\epsilon:(0,1)\to[0,\infty]$ defined by
	$$\epsilon(\alpha)=\inf\{\lambda>0;\, A\subset C(f,\alpha)-\lambda x_0\}\ \ \ \ (0<\alpha<1).$$  
	Then $\epsilon(\alpha)$ is $o(\alpha)$ as $\alpha\to 0^+$ if{f} $({-f})$ strongly exposes $A$ at $0$.
\end{lemma}

\begin{remark}\label{InfMin} Observe that if $\alpha\in(0,1)$ is such that $\epsilon(\alpha)$ is finite then, in the definition of the function $\epsilon$, the infimum is actually a minimum. Hence, in this case, we have that
	$A\subset C(f,\alpha)-\epsilon(\alpha) x_0.$
\end{remark}
\begin{proof}[Proof of Lemma~\ref{lemma:stronglyexposedVSopiccolo}] On the contrary, suppose that $\epsilon(\alpha)$ is not $o(\alpha)$ as $\alpha\to 0^+$, then there exist $M>0$ and $\alpha_n\to0^+$ such that $\epsilon(\alpha_n)>M\alpha_n$.
	Let $z_n\in A\setminus[C(f,\alpha_n)-M\alpha_n x_0]$ and observe that
	$$\textstyle f(z_n)+M{\alpha_n}=f(z_n+M\alpha_n x_0)<\alpha_n\|z_n+M\alpha_n x_0\|.$$
	Hence it holds
	$$\textstyle 0\leq f(z_n)<\alpha_n\|z_n+M\alpha_n x_0\|-M{\alpha_n}=\alpha_n(\|z_n+M\alpha_n x_0\|-{M}).$$
	Then $\|z_n+M\alpha_n x_0\|>{M}$ and hence eventually $\|z_n\|>\frac M2$.
	So, eventually we have
	$$\textstyle 0\leq f(\frac{z_n}{\|z_n\|})<\alpha_n\frac{\|z_n+M\alpha_n x_0\|-M}{\|z_n\|}\leq \alpha_n\frac{\|z_n\|+M\alpha_n-M}{\|z_n\|}\leq \alpha_n.$$
	In particular, we have $f(\frac{Mz_n}{2\|z_n\|})\to0$ as $n\to\infty$. Since $A$ is convex and $0\in A$, we have that eventually $\frac{Mz_n}{2\|z_n\|}\in A$, and hence that ${-f}$ does not strongly expose $A$ at $0$.

	For the other implication, suppose that $\epsilon(\alpha)$ is $o(\alpha)$ as $\alpha\to 0^+$. By Remark~\ref{InfMin},   we have that eventually (for $\alpha\to 0^+$) $\epsilon(\alpha)$ is finite and $$A\subset C(f,\alpha)-\epsilon(\alpha)x_0.$$  Let $x\in A\cap\{x\in X;\, f(x)\leq\alpha^2\}$, then eventually
	$$\alpha\|x+\epsilon(\alpha)x_0\|\leq f(x+\epsilon(\alpha)x_0)= f(x)+\epsilon(\alpha)f(x_0)\leq \alpha^2+\epsilon(\alpha)$$
	and hence $\|x\|\leq\frac{\epsilon(\alpha)}{\alpha}+\epsilon(\alpha)+\alpha$. This proves that $({-f})$ strongly exposes $A$ at $0$. 
\end{proof}

In the following two lemmas we analyse some relations between the  Attouch-Wets convergence of a sequence of sets and the containment of the sets of the sequence in a cone of the form $V(f,\alpha)$ or $C(f,\alpha)$.

\begin{lemma}\label{lemma:definitivamentenellaconca} Let  $B, B_n$ ($n\in\N$) be closed convex sets in $X$ such that $B_n\rightarrow
	B$ for the Attouch-Wets convergence, and $f\in S_{X^*}$. Suppose that $x_0\in S_X$ is such that $f(x_0)=1$ and suppose that $0\in B\subset \{x\in X;\, f(x)\leq 0\}$.
	Then, for each $\alpha\in(0,1)$ and $\epsilon>0$, there exists $n_0\in \N$ such that $B_n\subset V(f,\alpha)+\epsilon x_0$, whenever $n\geq n_0$.
\end{lemma}

\begin{proof}
	On the contrary, suppose that there exists a sequence of integers $\{n_k\}$ such that, for each $k\in\N$, there exists $$b_{n_k}\in B_{n_k}\setminus[V(f,\alpha)+\epsilon x_0].$$
	Since
	$$\mathrm{dist}(B, C(f,\alpha)+\epsilon x_0)>0,$$
	by Fact~\ref{fact:AW}, we can suppose without any loss of generality that $\|b_{n_k}\|\geq1$ ($k\in\N$).
		Since $b_{n_k}\not\in V(f,\alpha)+\epsilon x_0$, we have 
	$$f(b_{n_k})>\alpha\|b_{n_k}-\epsilon x_0\|+\epsilon\geq\alpha\|b_{n_k}\|.$$
	Let $\delta=\min\{\epsilon,\alpha/2\}$, since $0\in B$ and $B_n\rightarrow
	B$ for the Attouch-Wets convergence, we can suppose without any loss of generality that, for each $k\in\N$, there exists $d_k\in (\delta B_X)\cap B_{n_k}$. Let 
	$$\textstyle w_k=\frac1{\|b_{n_k}\|}b_{n_k}+\frac{\|b_{n_k}\|-1}{\|b_{n_k}\|}d_k\in B_{n_k},$$
	and observe that $\|w_k\|\leq1+\epsilon$. Moreover, we have
	$$\textstyle f(w_k)\geq f(b_{n_k})\frac1{\|b_{n_k}\|}-\|d_k\|\geq \alpha -\|d_k\|\geq \frac\alpha2.$$
	Since $\{w_k\}$ is a bounded sequence, by Fact~\ref{fact:AW}, $\mathrm{dist}(w_k,B)\to0$.
	Hence we get a contradiction since $\{w_k\}\subset \{x\in X;\, f(x)\geq\alpha/2\}$  and 
	$$\mathrm{dist}(B, \{x\in X;\, f(x)\geq\alpha/2\})>0.$$
\end{proof}

\begin{lemma}\label{lemma:definitivamentenelcono} Let  $A, A_n$ ($n\in\N$) be closed convex sets in $X$ such that $A_n\rightarrow
	A$ for the Attouch-Wets convergence, $f\in S_{X^*}$, $\alpha\in(0,1)$, and $\epsilon>0$. Suppose that $x_0\in S_X$ is such that $f(x_0)=1$ and suppose that $0\in A\subset C(f,\alpha)-\epsilon x_0$.
	Then, for each $\beta\in(0,\alpha)$ and $\epsilon'>\epsilon$, there exists $n_0\in \N$ such that $A_n\subset C(f,\beta)-\epsilon' x_0$, whenever $n\geq n_0$.
\end{lemma}

\begin{proof}
	
	Suppose on the contrary that there exists a sequence of integers $\{n_k\}$ such that, for each $k\in\N$, there exists $$a_{n_k}\in A_{n_k}\setminus[C(f,\beta)-\epsilon' x_0].$$
	Since $a_{n_k}+\epsilon' x_0\not\in C(f,\beta)$, we have 
\begin{equation}\label{eq:fuoridalcono}  f(a_{n_k}+\epsilon' x_0)=f(a_{n_k})+\epsilon'<\beta\|a_{n_k}+\epsilon' x_0\|.
\end{equation}	
			Fix any $\gamma\in(\beta,\alpha)$ and let $M\geq1$ be such that $M>\frac{2\epsilon'}{\alpha-\gamma}$. Finally, let $\theta\in (0,1)$ be such that
	\begin{enumerate}
		\item[(a)] $M-\theta>\frac{2\epsilon'}{\alpha-\gamma} $; 
\item[(b)]	$\frac{\beta M+\theta}{M-\theta}\leq\gamma$.
\end{enumerate}
Since 
$$\mathrm{dist}\bigl(C(f,\alpha)-\epsilon x_0,V(f,\beta)-\epsilon' x_0\bigr)>0,$$
by Fact~\ref{fact:AW}, we can suppose without any loss of generality that $\|a_{n_k}\|\geq M$ ($k\in\N$).
Moreover, since $0\in A$ and $A_n\rightarrow
A$ for the Attouch-Wets convergence, we can suppose without any loss of generality that, for each $k\in\N$, there exists $c_k\in A_{n_k}\cap\theta B_X$. 
Put, for each $k\in\N$, 
$$\textstyle b_k=\frac M{\|a_{n_k}\|}a_{n_k}+\frac{\|a_{n_k}\|-M}{\|a_{n_k}\|}c_k\in A_{n_k},$$
and observe that $M-\theta\leq\|b_k\|\leq M+\theta$.
Now, by (\ref{eq:fuoridalcono}), we have $f(a_{n_k})<\beta\|a_{n_k}\|$ and hence
$$f(b_k)\leq M\beta+\theta\leq\|b_k\|\frac{M\beta+\theta}{\|b_k\|}\leq\frac{M\beta+\theta}{M-\theta}\|b_k\|\leq\gamma\|b_k\|.$$
Moreover, since $\{b_k\}$ is bounded and $A\subset C(f,\alpha)-\epsilon x_0$, by Fact~\ref{fact:AW}, we have that eventually
$f(b_k)\geq\alpha\|b_k\|-2\epsilon'$ and hence that
$$\alpha\|b_k\|-2\epsilon'\leq f(b_k)\leq\gamma \|b_k\|.$$
In particular, we have that eventually $\|b_k\|\leq\frac{2\epsilon'}{\alpha-\gamma}<M-\theta$, a contradiction since $\|b_k\|\geq M-\theta$.   
\end{proof}

%
%

\section{The case where the intersection of limits sets is a singleton}\label{section:Infinite-dimensional}

In the sequel of the paper, we suppose that $X$ is a real Hilbert space. If $u,v\in X\setminus\{0\}$, we denote as usual
$$\textstyle\cos(u,v)=\frac{\langle u,v\rangle}{\|u\|\|v\|},$$
where $\langle u,v\rangle$ denotes the inner product between $u$ and $v$. 

If $K$ is a nonempty closed convex subset of $X$, let us denote by $P_K$ the projection onto the set $K$.
Several times without mentioning it, we shall use the variational characterization of best approximations from convex sets in Hilbert spaces: let $K$ be as above, $x\in X$ and $y_0\in K$, then 
$y_0=P_K(x)$ if and only if 
\begin{equation}\label{eq:proiezionesuconvesso}
\langle x-y_0,y-y_0\rangle\leq0 \ \ \ \ \text{whenever} \ y\in K.
\end{equation}
It is easy to see that,  if $x\not\in K$, (\ref{eq:proiezionesuconvesso}) is equivalent to the following condition:
\begin{equation}\label{eq:proiezionecoseno}
\|y-y_0\|\leq \|x-y\|\cos(y_0-y,x-y) \ \ \ \ \text{whenever} \ y\in K\setminus\{y_0\}.
\end{equation}
Moreover, if $K$ is a subspace of $X$ then (\ref{eq:proiezionesuconvesso}) becomes
\begin{equation}\label{eq:proiezionesusottospazio}
\langle x-y_0,y-y_0\rangle=0 \ \ \ \ \text{whenever} \ y\in K.
\end{equation}

Let us recall the definition of stability for a couple $(A,B)$ of subsets of $X$.

\begin{definition}
	Let  $A$ and $B$ be closed convex subsets of $X$ such that $A\cap B$ is nonempty. We say that the that
	the couple $(A,B)$ is {\em stable} if for each choice  of sequences $\{A_n\},\{B_n\}\subset\C(X)$ converging for the Attouch-Wets convergence to $A$ and $B$, respectively,   and for each choice of  the starting point $a_0$, the  corresponding perturbed alternating projections sequences  $\{a_n\}$ and $\{b_n\}$ converge in norm.
\end{definition}

\begin{remark}
We remark that in the above definition we can equivalently require that there exists $c\in A\cap B$ such that $a_n,b_n\to c$ in norm.
\end{remark}

\begin{proof}
	It suffices to  prove that if the perturbed alternating projections sequences  $\{a_n\}$ and $\{b_n\}$ converge in norm then they both converge to a point in $A\cap B$. 
	
	Let us start by proving that if $a_n\to a$ then $a\in A\cap B$. It is not difficult to prove that, since $$a_{n+1}= P_{A_n}P_{B_n}a_n=P_{A}P_{B}a_n+(P_{A_n}P_{B_n}-P_{A}P_{B})a_n$$  and since $A_n\to A,B_n\to B$  for the Attouch-Wets convergence, we have $a=P_AP_B a$. By \cite[Facts~1.1, (ii)]{BauschkeBorwein93}, we have that $a\in A\cap B$. Similarly, it is easy to see that 
	$$b_{n+1}=P_{B_n}a_n=P_{B}a_n+(P_{B_n}-P_{B})a_n\to P_Ba=a,$$
	and the proof is concluded.

\end{proof}

The main aim of this section is to prove that under the assumption that the sets
$A$ and $B$ are separated by a strongly exposing functional $f$  for the set $A$ (i.e. condition (i) in the introduction) the couple $(A,B)$ is stable.
The following theorem is the main result of this section.

\begin{theorem}\label{puntolur} Let $X$ be a Hilbert space and $A,B$
 nonempty closed convex subsets of $X$. Let $\{A_n\}$ and $\{B_n\}$
be two sequences of closed convex sets such that $A_n\rightarrow
A$ and $B_n\rightarrow B$ for the Attouch-Wets convergence. Suppose that there exist $y\in A\cap B$ and a linear continuous functional $f\in S_{X^*}$ such that $\inf f(B)=f(y)=\sup f(A)$  and such that $f$ strongly exposes $A$ at $y$.
Then, for each $a_0\in X$, the corresponding perturbed alternating projections sequences  $\{a_n\}$ and $\{b_n\}$ (with starting point $a_0$), converge to $y$ in norm.
\end{theorem}

Before starting with the proof of the theorem we need some preliminary work. First of all, let us observe that without any loss of generality we can suppose that $y=0$ and hence that $$\inf f(A)=f(0)=\sup f(B).$$

Suppose that $x_0\in S_X$ is such that $f(x_0)=1$, i.e., $f$ is represented by $x_0$, in the sense that $f(\cdot)=\langle x_0,\cdot\rangle$. Then it is straightforward to give the following representation of the cones $C(f,\alpha)$ and $V(f,\alpha)$, introduced  at the beginning of Section~\ref{Notations and preliminaries}: if we define 
$$\textstyle C(\theta):=\{x\in X\setminus\{0\};\, \cos(x,x_0)\geq\sin(\theta)\}\cup\{0\}\ \ \ (\theta\in(0,\frac\pi2)),$$ 
 then  the set $C(\theta)$  coincides with $C(f,\sin\theta)$. 
Similarly, if we define 
$$\textstyle V(\theta):=\{x\in X\setminus\{0\};\, \cos(x,x_0)\leq\sin(\theta)\}\cup\{0\}\ \ \ (\theta\in(0,\frac\pi2)),$$ 
 then the set $V(\theta)$ coincides with $V(f,\sin\theta)$. We shall need the following simple fact.

\begin{fact}\label{fact:coseno} Suppose that $\theta_1,\theta_2\in(0,\frac\pi2)$ are such that $\theta_1<\theta_2$. If $x\in C(\theta_2)\setminus\{0\}$ and $y\in V(\theta_1)\setminus\{0\}$ then $\cos(x,y)\leq\cos(\theta_2-\theta_1)$.
\end{fact}

\begin{proof}
For $z\in X\setminus\{0\}$ let us denote $\theta_z=\frac\pi2-\arccos \cos(z,x_0)$ and observe that
$$z\in C(\theta_2)\Leftrightarrow \theta_z\geq\theta_2\ \ \ \text{and}\ \ \ z\in V(\theta_1)\Leftrightarrow \theta_z\leq\theta_1.$$ 
Let us define $x_1=x-f(x)x_0$ and $y_1=y-f(y)x_0$, then

$$\textstyle \cos(x,y)\leq\frac{f(x)f(y)}{\|x\|\|y\|}+\frac{\|x_1\|\|y_1\|}{\|x\|\|y\|}=\cos(\theta_x-\theta_y)\leq\cos(\theta_2-\theta_1).$$
\end{proof}

\begin{proof}[Proof of Theorem~\ref{puntolur}]

Fix $M>0$, it suffices to prove that  the sequences $\{a_n\}$ and $\{b_n\}$ are eventually contained in $2M B_X$.
Let  $f\in S_{X^*}$ and $x_0\in X$ be as above.  Let $\alpha\in (0,1)$ and let
$$\epsilon(\alpha)=\inf\{\lambda>0;\, A\subset C(f,\alpha)-\lambda x_0\}\in[0,\infty],$$ 
 by Lemma~\ref{lemma:stronglyexposedVSopiccolo}, $\epsilon(\alpha)$ is $o(\alpha)$ as $\alpha\to 0^+$. In particular, we can fix $\beta\in(0,1/3)$ such that if  $\theta=\frac12\arcsin (2\beta)$ then  $\epsilon':=2\epsilon(3\beta)\in\R$ and
\begin{enumerate}
	\item[(a)] $\epsilon'\leq M/2$;
	\item[(b)] $\sin\theta+\frac8M\epsilon'\leq\sin(\frac43\theta)$;
	\item[(c)] $\sin(2\theta)-\frac8M{\epsilon'}\geq\sin(\frac53\theta)$;
	\item[(d)] $\cos(\frac13\theta)+\frac2M{\epsilon'}\leq\cos(\frac16\theta)$.
\end{enumerate}  
 Since, by Remark~\ref{InfMin}, $0\in A\subset C(f,3\beta)-\epsilon(3\beta) x_0$, by Lemma~\ref{lemma:definitivamentenelcono}, we have that eventually $$A_n\subset C(f,2\beta)-2\epsilon(3\beta) x_0=C(2\theta)-\epsilon' x_0.$$ 	
Since, $0\in B\subset\{x\in X;\, f(x)\leq0\}$, by Lemma~\ref{lemma:definitivamentenellaconca}, we have that eventually $$B_n\subset V(\theta)+\epsilon' x_0.$$ 	 
Since $0\in A\cap B$, $A_n\rightarrow
A$ and $B_n\rightarrow B$ for the Attouch-Wets convergence, eventually there exist  $x_n\in A_n\cap \epsilon'B_X$ and $y_n\in B_n\cap \epsilon'B_X$. 

\begin{claim} Eventually, if $a_n,b_n,b_{n+1}\not\in M B_X$, the following conditions hold:
\begin{enumerate}
	\item  $a_n-x_n\in C(\frac53\theta)$;
	\item $b_n-x_n\in V(\frac43\theta)$;
	\item $a_n-y_{n+1}\in C(\frac53\theta)$;
	\item $b_{n+1}-y_{n+1}\in V(\frac43\theta)$.
\end{enumerate}
\end{claim}

\begin{proof}[Proof of the claim]
	Let us prove (i) and (ii), the proof of (iii) and (iv) is similar. To prove (i), observe that, since $a_n\in A_n\subset C(2\theta)-\epsilon' x_0$, we have
	\begin{eqnarray*}
		f(a_n-x_n)&\geq& f(a_n+\epsilon'x_0)-2\epsilon'\\      
		&\geq&\sin(2\theta)(\|a_n+\epsilon'x_0\|)-2\epsilon'\\
		&\geq&\sin(2\theta)(\|a_n-x_n\|-2\epsilon')-2\epsilon'\\
		\textstyle&=&\textstyle\|a_n-x_n\|(\sin(2\theta)-\frac{2\epsilon'\sin(2\theta)+2\epsilon'}{\|a_n-x_n\|})\\ 
		\textstyle&\geq&\textstyle\|a_n-x_n\|(\sin(2\theta)-\frac8M{\epsilon'}) \\
		&\geq&\textstyle\|a_n-x_n\|\sin(\frac53\theta),
	\end{eqnarray*}
	where the last inequality holds by (c).
	To prove (ii), we proceed similarly: observe that, since $b_n\in B_n\subset V(\theta)+\epsilon' x_0$, we have
	\begin{eqnarray*}
		f(b_n-x_n)&\leq& f(b_n-\epsilon'x_0)+2\epsilon'\\      
		&\leq&\sin(\theta)(\|b_n-\epsilon'x_0\|)+2\epsilon'\\
		&\leq&\sin(\theta)(\|b_n-x_n\|+2\epsilon')+2\epsilon'\\
		\textstyle&=&\textstyle\|b_n-x_n\|(\sin\theta+\frac{2\epsilon'\sin\theta+2\epsilon'}{\|b_n-x_n\|})\\ 
		\textstyle&\leq&\textstyle\|b_n-x_n\|(\sin\theta+\frac8M{\epsilon'}) \\
		&\leq&\textstyle\|b_n-x_n\|\sin(\frac43\theta),
	\end{eqnarray*}
	where the last inequality holds by (b). The claim is proved.
	\end{proof} 

Now, since $a_n=P_{A_n}b_n$ and $x_n\in A_n$, by (\ref{eq:proiezionecoseno}), it holds
\begin{equation}\label{eq:normadecresce}
\|a_n-x_n\|\leq\|b_n-x_n\|\cos(a_n-x_n,b_n-x_n).
\end{equation}  
Then we can observe that, by (i) and (ii) in our claim and by Fact~\ref{fact:coseno}, we have that eventually, if $a_n,b_n\not\in M B_X$, it holds $\|a_n-x_n\|\leq\|b_n-x_n\|\cos(\frac13\theta)$ and hence 
$$\textstyle \|a_n\|\leq\|a_n-x_n\|+\epsilon'\leq(\|b_n\|+\epsilon')\cos(\frac13\theta)+\epsilon'\leq\|b_n\|(\cos(\frac13\theta)+\frac2M\epsilon')\leq\|b_n\|\cos(\frac16\theta),$$
where the last inequality holds by (d).
Similarly, since $b_{n+1}=P_{B_n}a_n$ and $y_{n+1}\in B_n$, it holds $\|b_{n+1}-y_{n+1}\|\leq\|a_n-y_{n+1}\|\cos(b_{n+1}-y_{n+1},a_n-y_{n+1})$. By (iii) and (iv) in our claim and by Fact~\ref{fact:coseno}, we have that eventually, if $a_n,b_{n+1}\not\in M B_X$, it holds $\|b_{n+1}-y_{n+1}\|\leq\|a_n-y_{n+1}\|\cos(\frac13\theta)$ and hence 
$$\textstyle \|b_{n+1}\|\leq(\|a_n\|+\epsilon')\cos(\frac13\theta)+\epsilon'\leq\|a_n\|(\cos(\frac13\theta)+\frac2M\epsilon')\leq\|a_n\|\cos(\frac16\theta),$$
where the last inequality holds by (d).

By (\ref{eq:normadecresce}) and by the observations above, there exists $n_0\in\N$ such that if $n\geq n_0$ then the following conditions hold:
\begin{enumerate}
	\item[($\alpha)$] if $a_n,b_n\not\in M B_X$ then $ \|a_n\|\leq\|b_n\|\cos(\frac16\theta)$, and if $a_n,b_{n+1}\not\in M B_X$ then $\|b_{n+1}\|\leq\|a_n\|\cos(\frac16\theta)$;
	\item[($\beta)$] if $b_n\in M B_X$ then $\|a_n\|\leq\|b_n\|+2\epsilon'\leq 2M$, and if $a_n\in M B_X$ then $\|b_{n+1}\|\leq\|a_n\|+2\epsilon'\leq 2M$.
	 \end{enumerate}

Now, it is easy to see that there exists $n_1\geq n_0$ such that $a_{n_1}\in MB_X$ or $b_{n_1}\in MB_X$. Indeed, since $\cos(\frac16\theta)<1$, the fact that, for each $n\geq n_0$, $a_n,b_n\not\in M B_X$ contradicts ($\alpha$). By ($\beta$) and taking into account also  ($\alpha$), we obtain that $a_n,b_n\in2M B_X$, whenever $n>n_1$.
\end{proof}

\begin{corollary}\label{corollary:puntolur} Let $X$ be a Hilbert space, $B$
	a nonempty closed convex subset of $X$, $A$ a body in $X$ and
	$y\in\partial A$ an LUR point of $A$. Let $\{A_n\}$ and $\{B_n\}$
	be two sequences of closed convex sets such that $A_n\rightarrow
	A$ and $B_n\rightarrow B$ for the Attouch-Wets convergence.
	Suppose that $A\cap B=\{y\}$.
	Then, for each $a_0\in X$, the corresponding perturbed alternating projections sequences  $\{a_n\}$ and $\{b_n\}$ (with starting point $a_0$), converge to $y$ in norm.
\end{corollary}

\begin{proof}
 Since $(\inte A)\cap B=\emptyset$, by the Hahn-Banach separation theorem, there exists $f\in S_{X^*}$ such that $$\inf f(A)=f(y)=\sup f(B).$$
 Since $y$ is an LUR point of $A$, by Lemma~\ref{slicelimitatoselur}, $f$ strongly exposes $A$ at $y$. The thesis follows by Theorem~\ref{puntolur}.
\end{proof}

It is worth noting that, in the recent paper \cite{GrKaKuReich},  a result concerning the convergence of iterates of nonexpansive mapping has been obtained under a geometrical condition involving LUR points.

\section{The case where the interior of the intersection of limits sets is nonempty} \label{Section:nonempty interior}

The main aim of this section is to prove that, under the assumption that the interior of  $A\cap B$ is nonempty, the couple $(A,B)$ is stable.

We start by the following two dimensional fact. Even if the argument used is elementary we include a sketch of a possible proof for the sake of completeness. 

\begin{fact}\label{fact:norms}
	Let $X$ be a Hilbert space and $\epsilon,K>0$. Then there exists a constant $\mu>0$ such that, whenever $C$ is a closed convex subset  of $X$ containing $\epsilon B_X$ and $x\in KB_X$, we have
\begin{equation}\label{eq:angleboundedawayfrom0}
\|x-P_Cx\|\leq\mu(\|x\|-\|P_C x\|).
\end{equation} 
\end{fact}

\begin{proof} We claim that $\mu=K/\epsilon$ works. Let us denote by $\theta(u,v)$ the angle between two not null  vectors $u$ and $v$.

	Let us denote $y=P_C x$.
We can (and do) assume  that $y$ and $x$ are not proportional  (if else (\ref{eq:angleboundedawayfrom0}) trivially holds). Hence, since $\epsilon B_X\subset C$, we have that $\epsilon<\|y\|<\|x\|$. Let $Y=\mathrm{span}\{x,y\}$ and let $w\in\epsilon S_Y$ be such that:
\begin{enumerate}
	\item the line containing $\{y,w\}$ is tangent to $\epsilon B_Y$;
	\item the segment $[y,w]$ intersects the segment $[0,x]$.
\end{enumerate}  
Observe that existence of such an element $w$ is guaranteed by the fact that $\|x-y\|\leq \|x\|-\epsilon$.
  Since the vectors $w$ and $w-y$ are orthogonal, we clearly have  $\sin\theta(-y,w-y)\geq \epsilon/K $.  Let us denote $z=\frac{\|y\|}{\|x\|}x$, by the variational characterization of best approximations from convex sets in Hilbert spaces and by the fact that $\|z\|=\|y\|$, we have:
\begin{enumerate}
\item $\theta(x-y,w-y)\geq \pi/2$;
\item $\theta(-y,z-y)\leq \pi/2$.
\end{enumerate} 	
	It follows that $\theta(x-y,z-y)\geq \theta(-y,w-y)$ and hence that $$\|x-y\|\leq\frac{K}{\epsilon}\|x-z\|=\frac{K}{\epsilon}(\|x\|-\|y\|)$$
	\end{proof}

The following theorem is the main result of this section and it is an application of the previous argument.

\begin{theorem}\label{theorem:corpilur} Let $X$ be a Hilbert space and $A,B$
	nonempty closed convex subsets of $X$. 
	Suppose that $\inte(A\cap B)\neq\emptyset$, then the couple $(A,B)$ is stable.
\end{theorem}

\begin{proof}
	Without any loss of generality, we can suppose that $0\in\inte (A\cap B)$.
	Let $\{A_n\}$ and $\{B_n\}$
	be two sequences of closed convex sets such that $A_n\rightarrow
	A$ and $B_n\rightarrow B$ for the Attouch-Wets convergence. Suppose that $\{a_n\}$ and $\{b_n\}$ are
	the corresponding perturbed alternating projections sequences  with respect to a given starting point $a_0$.
	
	By Proposition 27 in \cite{PenotZalinescu} we have that $A_n\cap B_n \rightarrow A\cap B$ for the Attouch-Wets convergence. Hence, by Theorem 7.4.2 in  \cite{Beer}, we can suppose without any loss of generality that there exists $\epsilon>0$ such that  $\epsilon B_X\subset A_n\cap B_n $, whenever $n\in \N$. Since $0\in A_n\cap B_n$, we have that $\|a_n\|\leq\|b_n\|, \|b_n\|\leq\|a_{n-1}\|$ and hence   there exists $K>0$ such that $\{a_n\},\{b_n\}\subset K B_X$. By Fact~\ref{fact:norms}, we have that there exists $\mu>0$ such that $\|a_n-b_n\|\leq\mu(\|b_n\|-\|a_n\|)$ and
$\|b_n-a_{n-1}\|\leq\mu(\|a_{n-1}\|-\|b_n\|)$. Hence 
$$\textstyle \sum_{n=1}^N(\|a_n-b_n\|+\|b_n-a_{n-1}\|)\leq\sum_{n=1}^N\mu(\|a_{n-1}\|-\|a_n\|)=\mu(\|a_{0}\|-\|a_N\|).$$
This proves that the series $\sum_{n\in\N}(a_n-a_{n-1})$ is absolutely convergent and hence convergent, i.e., the sequence  $\{a_n\}$ is convergent. Similarly, we have that also the sequence $\{b_n\}$ is convergent and the proof is complete. 	
\end{proof}

By combining the results contained in Section~\ref{section:Infinite-dimensional} and the previous theorem we have the following corollary. This corollary describes the stability property for the couple $(A,B)$ where $A$ and $B$ are bodies.

\begin{corollary}\label{corollary:corpilur} Let $X$ be a Hilbert space, suppose that at least one of the following conditions holds.
	\begin{enumerate}
		\item $A$ is a closed convex set with nonempty interior, $f\in X^*\setminus\{0\}$ is such that $f$ strongly exposes $A$ at the origin, and  $B=\{x\in X;\, f(x)\geq \alpha\}$, where $\alpha\leq 0$.
       \item 	$A,B$ are bodies in $X$ such that $A$ is LUR and $A\cap B\neq\emptyset$.
\end{enumerate}  Then the couple $(A,B)$ is stable.
	\end{corollary}

\begin{proof}
	(i) If $\alpha<0$ then $\inte(A\cap B)\neq\emptyset$ and we can apply Theorem~\ref{theorem:corpilur}. If $\alpha=0$ apply Theorem~\ref{puntolur}.
	
	\noindent (ii) If $\inte(A\cap B)\neq\emptyset$ we can apply Theorem~\ref{theorem:corpilur}. If $\inte(A\cap B)=\emptyset$, since $A$ and $B$ are bodies, we have that 
	$\inte(A)\cap B=\emptyset$. Since $A$ is an LUR body, there exists $y\in\partial A$ such that $A\cap B=\{y\}$. Apply Corollary~\ref{corollary:puntolur}. 
\end{proof}

It is worth to remark that the assumptions (i) and (ii) in Corollary \ref{corollary:corpilur} cannot be avoided if we ask for a stable couple of bodies. Indeed, when we consider two bodies with nonempty intersection, the typical situation in which (i) and (ii) fail is the following:  there exists a functional $f\in X^*\setminus\{0\}$ separating the bodies $A$ and $B$ but $f$ strongly exposes neither $A$ nor $B$.  The following simple 2-dimensional example shows that, in general, in this case we cannot guarantee that the couple $(A,B)$ is stable.

\begin{example}\label{ex: notconverge}
 Let $X=\R^2$ and let us consider, for each $h\in\N$, the following subsets of $X$:
 \begin{eqnarray*}
 A&=&\textstyle \conv\{(1,1),(-1,1),(1,0),(-1,0)\};\\ 
C_{2h}&=&\textstyle\conv\{(1,1),(-1,1),(1,\frac1h),(-1,0)\};\\ 
 C_{2h-1}&=&\textstyle\conv\{(1,1),(-1,1),(1,0),(-1,\frac1h)\};\\
 B&=&\textstyle\conv\{(1,-1),(-1,-1),(1,0),(-1,0)\};\\ 
D_{2h}&=&\textstyle\conv\{(1,-1),(-1,-1),(1,-\frac1h),(-1,0)\};\\
D_{2h-1}&=&\textstyle\conv\{(1,-1),(-1,-1),(1,0),(-1,-\frac1h)\}.
 \end{eqnarray*} 
 
 We claim that the couple $(A,B)$ is not stable. To prove this, let us consider the starting point $z_0=(0,0)$ and observe that, if we consider the points $a_k=(P_{C_1} P_{D_1})^k z_0$, it is clear that there exists $N_1\in\N$ such that $$\textstyle \|a_{N_1}-(1,0)\|<\frac12.$$
  Define $A_n=C_1$ and $B_n=D_1$  whenever $1\leq n\leq N_1$. Similarly,  if we consider the points $a_{N_1+k}=(P_{C_2} P_{D_2})^{k} a_{N_1}$ then there exists $N_2\in\N$ such that $$\textstyle \|a_{N_1+N_2}-(-1,0)\|<\frac12.$$ Define $A_n=C_2$ and $B_n=D_2$ whenever $N_1+1\leq n\leq N_1+N_2$. Then, proceeding inductively, it is easy to construct sequences  $\{A_n\}$ and $\{B_n\}$ converging respectively to $A$ and $B$ for the Attouch-Wets convergence and such that the perturbed alternating projections sequences  $\{a_n\}$ and $\{b_n\}$, w.r.t. $\{A_n\}$ and $\{B_n\}$ and with starting point $z_0$, do not converge.
\end{example}

\subsection*{Inequality constraints}

Inequality constraints are a typical example of problem that can be solved by projections and reflections methods (see, e.g., \cite[Remark~3.17]{BorweinSimsTam}). It appears very in often in mathematical programming theory. This problem reveals to be a stable problem under mild assumptions. Indeed, in the rest of this section we will show that under suitable additional hypotheses also the method of perturbed alternating projections sequences can be applied to deal with such a problem. 

Given a closed convex cone $K$ in a Hilbert space $X$ (recall that a subset $K$ of $X$ is called cone if $\lambda k\in K$, whenever $\lambda\in [0,\infty)$ and $k\in K$), we denote by $K^-$ its {\em negative polar cone}, i.e., the closed convex cone defined by
$$K^-=\{x\in X;\, \langle x,k \rangle\leq 0, \ \text{whenever}\ k\in K\}.$$ 
Let us suppose that $a\in X\setminus\{0\}$, $b\in \R$, and define $A=\{x\in X;\, \langle a, x\rangle\leq b\}$. Then it is easy to observe that the following assertions hold true.
\begin{itemize}
	    \item If $\inte K\neq\emptyset$, $a_1,\ldots,a_n\in X$, $b_1,\ldots,b_n>0$ and $$B:=\{x\in X;\, \langle a_i, x\rangle\leq b_i,\ i=1,\ldots,n\}$$ then $\inte(B\cap K)\neq \emptyset$.
	\item If $\inte K\neq\emptyset$ and $a\not\in K^-$ then $\inte(A\cap K)\neq \emptyset$.
	\item  If $a\in\inte (K^-)$ and $b=0$ then $A$ and $K$ are separated by a strongly exposing functional for the set $K$.
\end{itemize}

\noindent Hence, by combining the previous observation,  Theorem~\ref{theorem:corpilur}, and Theorem~\ref{puntolur}, we obtain the following result about the convergence of perturbed projections for the inequality constraints problem. 

\begin{theorem}
Let $K$ be a closed convex cone in a Hilbert space $X$. Suppose that at least one of the following conditions holds true.
\begin{enumerate}
	\item $\inte K\neq\emptyset$, $a_1,\ldots,a_n\in X$, $b_1,\ldots,b_n>0$, and $$B:=\{x\in X;\, \langle a_i, x\rangle\leq b_i,\ i=1,\ldots,n\}.$$
	\item $\inte K\neq\emptyset$, $a\not\in K^-$, $b\in\R$, and $$B:=\{x\in X;\, \langle a, x\rangle\leq b\}.$$ 
	\item   $a\in\inte (K^-)$ and $$B:=\{x\in X;\, \langle a, x\rangle\leq 0\}.$$
\end{enumerate}
 Then the couple $(K,B)$ is stable.
\end{theorem}

\noindent As a corollary, we obtain the following finite-dimensional result, where the cone is the standard nonnegative lattice cone in $\mathbb{R}^N$.

\begin{corollary}
	Let $X=\R^N$ and $K=\{(x_k)_1^N\in\R^N;\, x_k\geq0, k=1,\ldots,N  \}$. Suppose that at least one of the following conditions holds true.
	\begin{enumerate}
		\item  $a_1,\ldots,a_n\in X$, $b_1,\ldots,b_n>0$, and $$B:=\{x\in X;\, \langle a_i, x\rangle\leq b_i,\ i=1,\ldots,n\}.$$
		\item  $a\not\in K^-$, $b\in\R$, and $$B:=\{x\in X;\, \langle a, x\rangle\leq b\}.$$ 
		\item   $a\in\inte (K^-)$ and $$B:=\{x\in X;\, \langle a, x\rangle\leq 0\}.$$
	\end{enumerate}
	Then the couple $(K,B)$ is stable.
\end{corollary}

\section{Perturbed alternating projections sequences for subspaces}\label{section:subspaces}

In this section, we study the convergence of the perturbed alternating projections sequences in the case where the limit sets are subspaces. The following elementary example shows that if the intersection of the subspaces is non-trivial, in general, convergence does not hold.

\begin{example}\label{ex:duedimensionale} Let $Z=\R^2$ and let us consider $A_n=A=B=\{(x,0)\in Z;\, x\in\R\}$ ($n\in\N$). For each $h\in\N$, let us consider the line $C_h=\{(x,\frac1h-\frac1{h^2}x);\,x\in\R\}$ passing through the points $(0,\frac1h)$ and $(h,0)$. Let us consider the starting point $z_0=(0,0)$ and observe that, if we consider the points $a_k=(P_A P_{C_1})^k z_0$, it is clear that there exists $N_1\in\N$ such that $\|a_{N_1}\|>\frac12$. Define $B_n=C_1$ whenever $1\leq n\leq N_1$. Similarly,  if we consider the points $a_{N_1+k}=(P_A P_{C_2})^{k} a_{N_1}$ then there exists $N_2\in\N$ such that $\|a_{N_1+N_2}\|>1$. Define $B_n=C_2$ whenever $N_1+1\leq n\leq N_1+N_2$. Then, proceeding inductively, it is easy to construct a sequence $\{B_n\}$ such that the perturbed alternating projections sequences  $\{a_n\}$ and $\{b_n\}$, w.r.t. $\{A_n\}$ and $\{B_n\}$ and with starting point $z_0$, are unbounded.
\end{example}

In order to avoid such a situation we consider the case in which the intersection of the subspaces reduces to the origin. We have the following theorem.

\begin{theorem}\label{prop:sottospazisommachiusa}
	Let $X$ be a Hilbert space and suppose that $U,V\subset X$ are closed subspaces such that $U\cap V=\{0\}$ and $U+V$ is closed. Let $\{A_n\}$ and $\{B_n\}$ be two sequences of closed convex sets such that $A_n\rightarrow
	U$ and $B_n\rightarrow V$ for the Attouch-Wets convergence.
	Then, for each $a_0\in X$, the corresponding perturbed alternating projections sequences  $\{a_n\}$ and $\{b_n\}$, with starting point $a_0$, converge to $0$ in norm.
\end{theorem} 

\noindent If $W$ is a subspace of $X$ and $\epsilon\in(0,1)$, let $W(\epsilon)\subset X$ be the set defined by
$$W(\epsilon)=\{w\in X\setminus\{0\};\,\exists u\in W\setminus\{0\}\ \text{such that}\ \cos(u,w)\geq1-\epsilon\}\cup\epsilon B_X.$$ 
An easy computation shows that:
\begin{equation}\label{eq:Uepsilon}
W(\epsilon)=\{w\in X\setminus\{0\};\,\exists u\in W\cap\|w\|S_X\  \text{such that}\  \|u-w\|^2\leq2\epsilon\|w\|^2\}\cup\epsilon B_X.
\end{equation}
Before starting with the proof of the theorem we need the following two lemmas.

\begin{lemma}\label{lemma:defnelconoSottospazio}
	Let $X$ be a Hilbert space and $U$ a subspace of $X$. Let $\{A_n\}$  be a sequence of closed convex sets such that $A_n\rightarrow
	U$ for the Attouch-Wets convergence. Then, for each $\epsilon\in(0,1)$, it eventually holds that $A_n\subset U(\epsilon)$.
\end{lemma}

\begin{proof} On the contrary, suppose  that  there exist $\epsilon\in(0,1)$ and a sequence  $\{n_k\}$ of integers such that, for each $k\in\N$, there exists $x_{n_k}\in A_{n_k}\setminus U(\epsilon)$. Since $A_n\rightarrow
	U$ for the Attouch-Wets convergence, we can suppose, without any loss of generality, that $\|x_{n_k}\|>1$ (indeed, we can observe that $\mathrm{dist}\bigl(U,X\setminus U(\epsilon)\bigr)>0$ and use Fact~\ref{fact:AW}). Let $\gamma\in(0,1)$ be such that $\frac{(1-\epsilon)(1+\frac\gamma{1-\epsilon})}{(1-\frac\epsilon2)(1-\gamma)}\leq1$ and let  $k\in\N$ be such that there exists $z_k\in A_{n_k}\cap\gamma B_X$. Consider  $$\textstyle w_k=\lambda x_{n_k}+(1-\lambda)z_k\in A_{n_k},$$
	where $\lambda=\frac1{\|x_{n_k}\|}$, and observe that $1-\gamma\leq\|w_k\|\leq1+\gamma$ and that, for each $u\in U$, we have
	\begin{eqnarray*}\textstyle \langle w_k,u\rangle&=&\lambda\langle x_{n_k},u\rangle+(1-\lambda)\langle z_{k},u\rangle\leq \|u\|(1-\epsilon)\|\lambda x_{n_k}\|+\gamma\|u\|\\
		&=&  \textstyle \|u\|(1-\epsilon)(1+\frac\gamma{1-\epsilon})\\
		&=& \textstyle
		[(1-\frac\epsilon2)\|u\|\|w_k\|]\frac{(1-\epsilon)(1+\frac\gamma{1-\epsilon})}{(1-\frac\epsilon2)\|w_k\|}\\
		&\leq& \textstyle 
		[(1-\frac\epsilon2)\|u\|\|w_k\|]\frac{(1-\epsilon)(1+\frac\gamma{1-\epsilon})}{(1-\frac\epsilon2)(1-\gamma)}\leq (1-\frac\epsilon2)\|u\|\|w_k\|.
	\end{eqnarray*}
	Hence, $w_k\in A_{n_k}\setminus U(\frac\epsilon2)$. Since $\{w_k\}$ is a bounded sequence, by Fact~\ref{fact:AW}, $\mathrm{dist}(w_k,U)\to0$.
	We get a contradiction since  
	$$\textstyle \mathrm{dist}\bigl(U, X\setminus 
	U(\frac\epsilon2)\bigr)>0.$$
\end{proof}

\begin{lemma}\label{lemma:approssimazioneprodotto}
	Let $U,V$ be closed subspace of a Hilbert space $X$ such that $U\cap V=\{0\}$ and $U+V$ is closed. Let $M\in(0,1)$, then there exist $\epsilon\in(0,M)$ and $\eta\in(0,1)$ such that, for each $x\in U(\epsilon)\setminus MB_X$, $y\in V(\epsilon)\setminus MB_X$ and $z\in\epsilon B_X$, we have $\cos(x-z,y-z)\leq\eta$.
\end{lemma}

\begin{proof}
	By \cite[Lemma~3.5]{FranchettiLight86}, we have that
	$$\textstyle \Omega:=\sup\{{<a,b>};\,a\in V\cap S_X,b\in U\cap S_X\}<1.$$
	Fix any $\eta\in(\Omega,1)$ and take $\epsilon\in(0,M)$ such that 
	$$\textstyle\bigl(\frac{M}{M-\epsilon}\bigr)^2\bigl(\Omega+ \frac{15\sqrt\epsilon}{M^2}\bigr)\leq\eta.$$ Suppose that 
	$x\in U(\epsilon)\setminus MB_X$, $y\in V(\epsilon)\setminus MB_X$ and $z\in\epsilon B_X$. By (\ref{eq:Uepsilon}), there exist $u\in U\cap \|x\|S_X$ and $v\in V\cap \|y\|S_X$ such that $\|x-u\|\leq\sqrt{2\epsilon}\|x\|$ and $\|y-v\|\leq\sqrt{2\epsilon}\|y\|$. Hence, $x':=x-u-z\in 3\sqrt\epsilon B_X$ and $y':=y-v-z\in 3\sqrt\epsilon B_X$. Then  we have:
	\begin{eqnarray*} 
		\textstyle \langle x-z,y-z \rangle &=& \langle u+x',v+y'\rangle \\
		&\leq&\textstyle  \langle u,v\rangle+\langle
		u,y'\rangle+\langle x',v\rangle+\langle x',y'\rangle\\
		&\leq&\textstyle \Omega\|x\|\|y\|+ 3\sqrt\epsilon\|x\|+3\sqrt\epsilon\|y\|+9\epsilon\\
		&\leq&\textstyle \|x\|\|y\|(\Omega+ \frac{3\sqrt\epsilon}{\|x\|}+\frac{3\sqrt\epsilon}{\|y\|}+\frac{9\epsilon}{\|x\|\|y\|})
		\\
		&\leq&\textstyle \|x\|\|y\|(\Omega+ \frac{6\sqrt\epsilon}{M}+\frac{9\epsilon}{M^2})\\
		&\leq&\textstyle \|x\|\|y\|(\Omega+ \frac{15\sqrt\epsilon}{M^2})\\
		&\leq&\textstyle \|x-z\|\|y-z\|\frac{\|x\|}{\|x\|-\epsilon}\frac{\|y\|}{\|y\|-\epsilon}(\Omega+ \frac{15\sqrt\epsilon}{M^2})\\
		&\leq&\textstyle \|x-z\|\|y-z\|\bigl(\frac{M}{M-\epsilon}\bigr)^2(\Omega+ \frac{15\sqrt\epsilon}{M^2})\\
		&\leq&\textstyle\eta \|x-z\|\|y-z\|.  
	\end{eqnarray*}
\end{proof}

\noindent We are now ready to prove our theorem.
\begin{proof}[Proof of Theorem~\ref{prop:sottospazisommachiusa}]
	
	Fix $M\in(0,1)$, it suffices to prove that eventually $a_n, b_n\in 3 M B_X$ (recall that $\{a_n\}$ and $\{b_n\}$ are defined as in Definition~\ref{def:perturbedseq}). Let $\epsilon\in(0,M)$ and $\eta\in(0,1)$ be given by Lemma~\ref{lemma:approssimazioneprodotto}. Let us consider the sets $U(\epsilon),V(\epsilon)$ and observe that, by Lemma~\ref{lemma:defnelconoSottospazio}, there exists $n_0\in\N$ such that if $n\geq n_0$ then $A_n\subset U(\epsilon)$ and $B_n\subset V(\epsilon)$. Let us fix $\epsilon'\in(0,\epsilon)$ such that $\eta+\frac{2\epsilon'}M\leq\frac{\eta+1}2$, then there exists an integer $n_1\geq n_0$ such that, for each $n\geq n_1$, there exist $x_n\in A_n\cap\epsilon' B_X$ and $y_n\in B_n\cap\epsilon' B_X$.
	
	Suppose that $n\geq n_1$,  we can observe that:
	\begin{itemize}
		\item by (\ref{eq:proiezionecoseno}) and Lemma~\ref{lemma:approssimazioneprodotto},  if $a_n,b_n\not\in M B_X$, it holds $\|a_n-x_n\|\leq\|b_n-x_n\|\eta$ and hence 
		$$\textstyle \|a_n\|\leq\|a_n-x_n\|+\epsilon'\leq\eta(\|b_n\|+\epsilon')+\epsilon'\leq\|b_n\|(\eta+\frac{2\epsilon'}M)\leq\frac{\eta+1}2\|b_n\|;$$
		\item similarly, if $a_n,b_{n+1}\not\in M B_X$, it holds 
		$$\textstyle \|b_{n+1}\|\leq\frac{\eta+1}2\|a_n\|;$$
		\item by (\ref{eq:proiezionecoseno}), if  $b_n\in M B_X$ then $\|a_n\|\leq\|b_n\|+2\epsilon'\leq 3M$ and, similarly, if $a_n\in M B_X$ then $\|b_{n+1}\|\leq3M$.
	\end{itemize}
	
	By the observations above and since $\frac{\eta+1}2<1$, proceeding as at the end of the proof of Theorem~\ref{puntolur}, it easily follows that eventually $a_n,b_n\in 3M B_X$.
\end{proof}

The remaining part of this section is devoted to proving that the assumption on the closedness of the sum of the subspaces, in Proposition~\ref{prop:sottospazisommachiusa}, cannot be removed. This result is contained in Theorem~\ref{teo:sommaNONchiusa} below and is inspired by the construction contained in \cite[Section~4]{FranchettiLight86}.   
Let $X=\ell_2$. For the sake of clearness, we point out that, in the sequel, we sometimes use the following notation: if, for each $h\in \N$, $x^h$ is an element of $X$, we denote by $\{x^h\}$ the corresponding sequence in $X$. Moreover, if $h\in\N$ is fixed, we can consider $x^h$ as a sequence of real numbers and we write $x^h=\{x^h_n\}_n$.
Now, suppose that $\{\theta_n\}\subset\R$ is a bounded sequence and  let us consider the linear continuous operator $D:X\to X$ given by $Dx=D\{x_n\}=\{\theta_n x_n\}$ ($x=\{x_n\}\in X$).
Suppose that $b=\{b_n\}\in X$ and consider the closed convex subsets of $Z=X\oplus_2X$ defined as follows:
$$
A=\{(x,0)\in Z;\,x\in X\}\ \ \ \ \text{and}\ \ \ \ V=\{(x,b+Dx)\in Z;\,x\in X\}.
$$ 
Observe that $A$ is a subspace of $Z$ and $V$ is an affine set in $Z$.
\begin{remark} \label{remark calcolo proiezioni}
	If $(\alpha,\beta)\in Z$
	then we obtain immediately that $P_A(\alpha,\beta)=(\alpha,0)$. 
	Now, let us suppose that $(\alpha,0)\in A$ and let us compute $P_V(\alpha,0)$. If we denote $P_V(\alpha,0)=(\{x_n\},\{b_n+\theta_n x_n\})$, by the characterization of best approximation in Hilbert space, we have, for each $\{y_n\}\in X$, 
	$$\textstyle \bigl\langle(\{x_n-\alpha_n\},\{b_n+\theta_n x_n\}),(\{y_n\},\{\theta_n y_n\})\bigr\rangle=0.$$
	Hence, we must have $x_n-\alpha_n+b_n\theta_n+x_n\theta^2_n=0$, whenever $n\in\N$. That is, for each $n\in\N$, it holds
	\begin{equation}\label{eq:proiez}
	\textstyle x_n=\frac{\alpha_n-\theta_n b_n}{1+\theta_n^2}.
	\end{equation} 
\end{remark}

\begin{lemma}\label{lemma:convAWpersottospazi}
	Let $Z$ be defined as above. Let $\{b^n\}\subset X$ be a norm null sequence. Let $D,D^n:X\to X$ ($n\in\N$) be linear bounded operators such that $D^n\to D$ in the operator norm. Then if we define
	$$W=\{(x,Dx)\in Z;\,x\in X\}\ \ \ \text{and}\ \ \ W_n=\{(x,b^n+D^n x)\in Z;\,x\in X\}\ \ \ (n\in\N)$$
	we have that $W_n\rightarrow W$ for the Attouch-Wets convergence.
\end{lemma}

\begin{proof}
	Let us fix $N\in\N$. If $z=(x,Dx)\in W\cap N B_Z$ then we can consider $z'=(x,b^n+D^nx)\in W_n$ and observe that $$\|z-z'\|_Z=\|Dx-D^nx-b^n\|_X\leq N\|D-D^n\|+\|b^n\|_X.$$
	Similarly, if $w=(y,b^n+D^ny)\in W_n\cap N B_Z$ then we can consider $w'=(y,Dy)\in W$ and observe that $$\|w-w'\|_Z=\|Dy-D^ny-b^n\|_X\leq N\|D-D^n\|+\|b^n\|_X.$$
	Hence, $h_N(W,W_n)\leq N\|D-D^n\|+\|b^n\|\to0$ ($n\to\infty$), and the proof is concluded.
\end{proof}

\begin{theorem}\label{teo:sommaNONchiusa} Let $Z$  be defined as above and $A=\{(x,0)\in Z;\,x\in X\}$, then there exist \begin{enumerate}
		\item[(a)] $B$ a closed subspace of $Z$,
		\item[(b)] $z_0\in Z$,
		\item[(c)] $\{A_n\},\{B_n\}\subset c(Z)$
		two sequences of sets converging to $A$ and $B$, respectively, for the Attouch-Wets convergence, 
	\end{enumerate} 
	such that the perturbed alternating projections sequences (w.r.t. $\{A_n\}$ and $\{B_n\}$ and with starting point $z_0$), are unbounded.
\end{theorem}

\begin{proof}
	Let us consider the sequence $\{a_n\}\subset\R$, given by  $a_n=4^{-n}$, and let us consider the operator $D:X\to X$, given by $D\{x_n\}=\{a_n x_n\}$. Then define $B=\{(x,Dx)\in Z;\,x\in X\}$ and, for each $n\in\N$, put $A_n=A$. Now, consider any $z_0=(\{\alpha_n\},0)\in A$ such that $\alpha_n>0$ ($n\in\N$) and $\|z_0\|<1$. 
	
	Let us put, $N_0=1$ and, for each $n\in\N$, $\alpha_n^{0,1}=\alpha_n$.  We shall define inductively (with respect to $h\in\N$) positive integers $N_h$, countable families of elements of $X$  $$\textstyle \{\alpha^{h,1}_n\}_n,\{\alpha^{h,2}_n\}_n,\{\alpha^{h,3}_n\}_n\ldots,$$  positive real numbers $M_h$, and  sets $C_h\subset Z$  such that:
	\begin{enumerate}
		\item \label{induzione M} $2^h+h>(1+M_h)^2\sum_{n=h+1}^\infty(\alpha_n^{h-1,N_{h-1}})^2> 2^h$
		\item \label{induzione C} $C_h=\{(x,b^h+D^hx)\in Z;\, x\in X\}$, where $D^h:X\to X$ is given by $D^h\{x_n\}=\{\theta^h_n x_n\}$ and where $b^h=\{b^h_n\}_n\in X$ and $\theta^h_n\in\R$ are given by 
		$$\textstyle b^h_n=\begin{cases}
		0\  \   \ &\text{if}\ n\leq h\\
		\alpha_n^{h-1,N_{h-1}} a_n\frac{1+M_h}{M_h} \   &\text{if}\ n> h
		\end{cases}\ \ \ \ \text{and}\ \ \ \  \
		\theta^h_n=\begin{cases}
		a_n\ &\text{if}\ n\leq h\\
		-\frac{1}{M_h}a_n\   &\text{if}\ n> h
		\end{cases}
		;$$
		\item \label{induzione P(h,1)} $(\{\alpha^{h,1}_n\}_n,0)=P_A P_{C_h}(\{\alpha^{h-1,N_{h-1}}_n\}_n,0)$;
		\item \label{induzione P(h,t+1)}  $(\{\alpha^{h,t+1}_n\}_n,0)=P_A P_{C_h}(\{\alpha^{h,t}_n\}_n,0)$, $t\in\N$;
		\item \label{induzione disuguaglianze} $2^h+h>\sum_{n=1}^\infty(\alpha_n^{h,N_{h}})^2\geq\sum_{n=h+1}^\infty(\alpha_n^{h,N_{h}})^2> 2^h$;
		\item  \label{induzione alfa} $\alpha^{h,t}_n>0$, whenever $n,t\in\N$.
	\end{enumerate}
	Let us show that this is possible.
	Let $h\in\N$ and suppose we already
	have  $N_{h-1}\in\N$ and sequences $$\{\alpha^{h-1,1}_n\}_n,\ldots,\{\alpha^{h-1,N_{h-1}}_n\}_n\subset X$$ such that the following conditions hold:
	\begin{itemize}
		\item $\textstyle 2^{h-1}+h-1>\sum_{n=1}^\infty(\alpha_n^{h-1,N_{h-1}})^2$;
		\item $\alpha^{h-1,N_{h-1}}_n>0$, whenever $n\in\N$.
	\end{itemize}
	
	(Observe that for $h=1$ the two conditions above are trivially satisfied since $\alpha_n^{0,N_0}=\alpha_n>0$ and $\sum_{n=1}^\infty(\alpha_n^{0,N_0})^2=\|z_0\|^2<1$.)
	
	By combining these two relations, we obtain that 
	$$
	2^{h}+h>\sum_{n=1}^\infty(\alpha_n^{h-1,N_{h-1}})^2>\sum_{n=h+1}^\infty(\alpha_n^{h-1,N_{h-1}})^2>0.
	$$
	Hence there exists a positive real number $M_h$ such that (\ref{induzione M}) holds true. Now, let us consider $C_h$ defined as in (\ref{induzione C}). Then, by the relations in (\ref{induzione P(h,1)}) and (\ref{induzione P(h,t+1)}), we define $\{\alpha^{h,t}_n\}_n$ ($t\in\N$). We just have to prove that there exists $N_h\in\N$  such that (\ref{induzione disuguaglianze}) is satisfied and that (\ref{induzione alfa}) holds true. By taking into account Remark \ref{remark calcolo proiezioni} and the fact that $(\{\alpha^{h,1}_n\}_n,0)=P_A P_{C_h}(\{\alpha^{h-1,N_{h-1}}_n\}_n,0)$, an easy computation shows that,  for each $n>h$,
	$$\textstyle \alpha^{h,1}_n=\alpha^{h-1,N_{h-1}}_n\frac{1+\frac{1+M_h}{M_h^2}a_n^2}{1+\frac1{M_h^2}a_n^2}.$$
	Repeating $N$ times the same argument yields:
	$$\textstyle \alpha^{h,N}_n=\alpha^{h-1,N_{h-1}}_n\frac{1+\frac{1+M_h}{M_h^2}a_n^2\sum_{l=0}^{N-1}(1+\frac1{M_h^2}a_n^2)^l}{(1+\frac1{M_h^2}a_n^2)^N}.$$

	Moreover, for each $n\leq h$,
	$$\textstyle \alpha^{h,1}_n=\alpha^{h-1,N_{h-1}}_n\frac{1}{1+a_n^2}.$$
	Repeating $N$ times the same argument yields:
	$$\textstyle \alpha^{h,N}_n=\alpha^{h-1,N_{h-1}}_n\frac{1}{(1+a_n^2)^N}.$$
	Since $$\textstyle \frac{1+\frac{1+M_h}{M_h^2}a_n^2\sum_{l=0}^{N-1}(1+\frac1{M_h^2}a_n^2)^l}{(1+\frac1{M_h^2}a_n^2)^N}=\frac{-M_h+(1+M_h)(1+\frac1{M_h^2}a_n^2)^N}{(1+\frac1{M_h^2}a_n^2)^N}\to 1+M_h \ \ (N\to \infty)$$
	and 
	$$
	\textstyle \frac{1}{(1+a_n^2)^N}\to 0\ \ (N\to \infty),
	$$
	by (\ref{induzione M}) we obtain that there exists $N_h\in\N$ such that
	$$\textstyle 2^h+h>\sum_{n=1}^\infty(\alpha_n^{h,N_{h}})^2\geq\sum_{n=h+1}^\infty(\alpha_n^{h,N_{h}})^2> 2^h.$$
	Moreover, it follows immediately that condition (\ref{induzione alfa}) is satisfied.

	Now, if $\sum_{k=0}^{h-1} N_k\leq n<\sum_{k=0}^{h} N_k$, put $B_n=C_h$. By our construction, it holds that $a_N=(\{\alpha_n^{h,N_h}\},0)$ where $N=\sum_{k=1}^{h} N_k$. In particular, $$\|b_N\|^2\geq\|P_{A} b_N\|^2=\|P_{A_N} b_N\|^2=\|a_N\|^2\geq \sum_{n=h+1}^\infty(\alpha_n^{h,N})^2> 2^h $$
	and hence the the sequences $\{a_n\}$ and $\{b_n\}$ are unbounded.  
	
	\smallskip
	
	It remains to prove that $B_n\rightarrow B$ for the Attouch-Wets convergence or, equivalently,  that $C_h\rightarrow B$ for the Attouch-Wets convergence. In view of Lemma~\ref{lemma:convAWpersottospazi}, it suffices to prove that the sequence $\{b^h\}$ is norm null and that $D^h\to D$ in the operator norm.
	
	By the inequalities in (\ref{induzione M}) and (\ref{induzione disuguaglianze}), we have 
	$$\textstyle (1+M_h)^2(2^{h-1}+h-1)\geq(1+M_h)^2\sum_{n=h+1}^\infty(\alpha_n^{h-1,N_{h-1}})^2> 2^h,$$
	and hence  
	$$\textstyle (1+M_h)^2>\frac{2^h}{2^{h-1}+h-1}.$$
	Therefore the sequence $\{M_h\}$ is bounded away from $0$. Hence, the sequences $\{\frac1{M_h}\}$ and $\{\frac{1+M_h}{M_h}\}$ are bounded above by a positive constant $K$. Then, by the definition of $b^h$ in (\ref{induzione C}), we have
	$$\textstyle \|b^h\|\leq Ka_h\|\{\alpha_n^{h-1,N_{h-1}}\}\|_X\leq \frac K{4^h}\|\{\alpha_n^{h-1,N_{h-1}}\}\|_X\leq\frac K{4^h}\sqrt{2^{h-1}+h-1},$$
	where the last inequality holds by (\ref{induzione disuguaglianze}). Moreover, by the definition of $\theta_n^h$ in (\ref{induzione C}), we have that
	$$\textstyle \|(D-D^h)x\|^2\leq \sum_{n=h+1}^\infty (a_n-\frac1{M_h}a_n)^2x_n^2\leq(1+K)^2a^2_{h+1}\|x\|^2\ \ \ \ \ (x=\{x_n\}\in X). $$
	Therefore, finally we obtain that 
	$$\textstyle \|D-D^h\|\leq (1+K)a_{h+1}.$$
	
\end{proof}

\section*{Acknowledgments.}
The research of the authors is partially
supported by GNAMPA-INdAM, Project GNAMPA 2018. The research of the second author is partially
supported  by the Ministerio de Ciencia, Innovaci\'on y Universidades (MCIU), Agencia Estatal de Investigaci\'on (AEI) (Spain) and Fondo Europeo de Desarrollo Regional (FEDER) under project PGC2018-096899-B-I00 (MCIU/AEI/FEDER, UE).
 The authors thank S.~Reich and E.~Molho for useful remarks that helped them in preparing this paper.

\end{document}